\documentclass[10pt]{article}

\usepackage{amsmath}
\usepackage{amssymb}
\usepackage{amsthm}
\usepackage{graphics}
\usepackage[latin1]{inputenc}
\usepackage[T1]{fontenc}

\usepackage{enumerate}
\usepackage{xcolor}

\newtheorem{thm}{Theorem}[section]
\newtheorem{coro}[thm]{Corollary}
\newtheorem{lemma}[thm]{Lemma}
\newtheorem{rem}[thm]{Remark}

\makeatletter

\newcommand*{\plim}[1][]{%
   \if\relax\detokenize{#1}\relax
      \def\next{\qopname\relax m{lim}}%
   \else
      \def\next{\qopname\newmcodes@ m{#1-lim}}%
   \fi
   \next
}

\newcommand*{\psum}[1][]{%
   \DOTSB
   \if\relax\detokenize{#1}\relax\else
      \operatorname{#1-}\mkern-\thinmuskip
   \fi
   \sum@\slimits@
}

\makeatother

\DeclareMathOperator\supp{supp}
\newcommand{\N}{\mathbb{N}}             
\newcommand{\C}{\mathbb{C}}             

\newcommand{\e}{\epsilon}

\newcommand{\half}{\frac{1}{2}}

\newcommand{\Section}[1]{\section{#1} \setcounter{equation}{0}}

\begin{document}

\title{{Local H\"older Stability in the Inverse Steklov and Calder\'on  Problems for Radial Schr\"odinger operators and Quantified Resonances}}
\author{Thierry Daud\'e \footnote{Research supported by the French National Research Projects AARG, No. ANR-12-BS01-012-01, and Iproblems, No. ANR-13-JS01-0006} $^{\,1}$, Niky Kamran \footnote{Research supported by NSERC grant RGPIN 105490-2018} $^{\,2}$ and Fran{\c{c}}ois Nicoleau \footnote{Research supported by the French GDR Dynqua} $^{\,3}$\\[12pt]
 $^1$  \small Laboratoire de Math\'ematiques de Besan{\c{c}}on, UMR CNRS 6623, Universit\'e de Franche-Comt\'e, \\
 \small 25030, Besan{\c{c}}on, France. \\
\small Email: thierry.daude@univ-fcomte.fr \\
$^2$ \small Department of Mathematics and Statistics, McGill University,\\ \small  Montreal, QC, H3A 0B9, Canada. \\
\small Email: niky.kamran@mcgill.ca \\
$^3$  \small  Laboratoire de Math\'ematiques Jean Leray, UMR CNRS 6629, \\ \small 2 Rue de la Houssini\`ere BP 92208, F-44322 Nantes Cedex 03. \\
\small Email: francois.nicoleau@univ-nantes.fr }



\date{\today}


\maketitle


\begin{abstract}
We obtain H\"older stability estimates for the inverse Steklov and Calder\'on problems for Schr\"odinger operators corresponding to a special class of $L^2$ radial potentials on the unit ball. These results provide an improvement on earlier logarithmic stability estimates obtained in \cite{DKN5} in the case of the the Schr\"odinger operators related to deformations of the closed Euclidean unit ball. The main tools involve: i) A formula relating the difference of the Steklov spectra of the Schr\"odinger operators associated to the original and perturbed potential to the Laplace transform of the difference of the corresponding amplitude functions introduced by Simon \cite{Si1} in his representation formula for the Weyl-Titchmarsh function, and ii) A key moment stability estimate due to Still \cite{St}. It is noteworthy that with respect to the original Schr\"odinger operator, the type of perturbation being considered for the amplitude function amounts to the introduction of a finite number of negative eigenvalues and of a countable set of negative resonances which are quantified explicitly in terms of the eigenvalues of the Laplace-Beltrami operator on the boundary sphere. 


\vspace{0.5cm}




\end{abstract}

\tableofcontents

\Section{Introduction}
In a recent paper, \cite{DKN5}, we have obtained a set of logarithmic stability estimates in the inverse Steklov problem for the Laplace-Beltrami operator on a class of warped product Riemannian manifolds defined on a $d$-dimensional closed ball. These manifolds can be thought of as deformations of the closed Euclidean $d$-ball in which the deformation is parametrized by the choice of radial warping function. The deformations considered in \cite{DKN5} include both the regular and singular cases of warped product metrics \cite{Pe}. 
\vskip .1in
\noindent The approach taken in \cite{DKN5} was based on expressing the warped product metric in a coordinate system in which the metric takes the form of a conformal rescaling of the flat Euclidean metric. This enabled us by using the transformation law of the Laplace-Beltrami operator under conformal changes of metric to reformulate the inverse Steklov problem for the Laplace-Beltrami operator on the original deformed closed ball as the inverse Steklov problem for a Schr\"odinger operator on the Euclidean ball, with a potential expressed in terms of the warping function of the original deformed ball and its derivatives.  The logarithmic stability estimates that we obtained were thus of the nature of the estimates obtained by Alessandrini \cite{Ale} and Novikov \cite{Nov}. 
\vskip .1in
\noindent Our goal in the present paper is to improve the logarithmic stability results of \cite{DKN5} in a significant way by obtaining instead a set of \emph{H\"older} stability estimates and by highlighting {\it{explicitly}} the role played in these by resonances. 
As a byproduct, we also obtain new local H{\"o}lder stability estimates for the Calder\'on problem on the unit ball for a set of admissible radial potentials. In the approach followed in the present paper, instead of starting from a deformed closed $d$-ball, we reverse the initial step taken in \cite{DKN5} and start with a Schr\"odinger operator on the Euclidean $d$-ball, endowed with a radial potential. One reason for doing so is that the set of warped product metrics for which H\"older stability estimates can be obtained for the inverse Steklov problem will be significantly more restricted than the ones for which logarithmic stability estimates were obtained in \cite{DKN5}. Thus by considering Schr\"odinger operators, we are broadening the range of inverse Steklov problems to which our H{\"o}lder stability results will apply. We emphasize nevertheless that the class of potentials for which we are able to establish H{\"o}lder stability for the inverse Steklov problem is still rather special and that any extension of our results to more general potentials may require stronger techniques than the ones we are using the present paper. We now proceed to describe our main results.


\vskip .1in
\noindent We let $M$ denote the manifold with boundary given by the $d\geq 3$-dimensional closed Euclidean unit ball centered at the origin in ${\mathbb{R}}^d$, with boundary given by the unit sphere $S^{d-1}$. We shall often work in hyperspherical coordinates, which are not well defined at the origin, so that we shall occasionally commit a slight abuse of notation by writing
\begin{equation} \label{manifoldI}
	M = (0,1] \times S^{d-1}\,,
\end{equation} 	
We consider the Dirichlet problem for the Schr\"odinger operator with a potential $q$, given by
\begin{equation} \label{SchrodI}
	\left\{ \begin{array}{cc}
		-\triangle u + q\,u= 0, & \textrm{on} \ M\,, \\
		u = \psi \in H^{\half}(\partial M), & \textrm{on} \ \partial M\,,
	\end{array} \right.
\end{equation}

\vskip .1in
\noindent 
When $q\in L^{\infty}(M)$ and $\lambda=0$ is not a Dirichlet eigenvalue of the above Schr\"odinger operator, the Dirichlet problem (\ref{SchrodI}) has a unique solution $u\in H^{1}(M)$.  The Dirichlet-to-Neumann (DN) map $\Lambda_q$ is then (formally) defined as an operator from $H^{1/2}(\partial M)$ to $H^{-1/2}(\partial M)$ by
\begin{equation} \label{DNI}
	\Lambda_{q} \psi = \left( \partial_\nu u \right)_{|\partial M}\,,
\end{equation}
where $u$ is the unique solution of (\ref{SchrodI}) and $\left( \partial_\nu u \right)_{|\partial M}$ is the normal derivative of $u$ with respect to the outer unit normal vector $\nu$ on $\partial M$.

\vskip .1in
\noindent 
The DN map thus defined is a self-adjoint operator on $L^2(\partial{M}, dS_g)$, where $dS_g$ denotes the metric induced by the Euclidean metric on the boundary sphere $\partial M=S^{d-1}$. Its spectrum (the so-called {\it  Steklov spectrum}) is discrete  and accumulates at infinity. We shall denote the Steklov eigenvalues (counted with multiplicity) by
\begin{equation} \label{SteklovI}
	0 = \sigma_0 < \sigma_1 \leq \sigma_2 \leq \dots \leq \sigma_k \to \infty\,.
\end{equation}
The Steklov spectrum will be the central object of study in this paper.

\vskip .1in
\noindent 
For the remainder of this paper, we shall assume that the potential is radial and we write $q=q(r)$, where $r$ denotes the Euclidean distance to the origin. It will be convenient to replace the radial coordinate $r\in (0,1]$ by a new radial coordinate $x\in [0,\infty)$ defined by $x=-\log r$, in which case the boundary of $M$ now corresponds to $x=0$. The Euclidean metric then takes the form 

\[
g=f(x)^{4}(dx^2+d\Omega^2)\,,
\]
where $f(x)= \exp(-x/2)$ and $d\Omega^2$ denotes the round metric on the unit sphere $S^{d-1}$. The Dirichlet problem (\ref{SchrodI}) gets transformed into
\begin{equation} \label{SchrodnewI}
	\left\{ \begin{array}{cc}
		[-\partial_x^2 - \triangle_{S} + Q(x)] v = -\frac{(d-2)^2}{4} v, & \textrm{on} \ M\,, \\
		v = f^{d-2} \psi, & \textrm{on} \ \partial M\,,
	\end{array} \right.
\end{equation}
where $\triangle_{S}$ denotes the Laplacian on the boundary sphere $S^{d-1}$, where $Q(x):=e^{-2x}q(e^{-x})$ and where $v=f^{d-2}u$. In other words, if $(r,\omega)$ denote hyperspherical coordinates in ${\bf{R}}^d$, then we have $v(x,\omega)=e^{-({{d-2}\over{2}})x}u(e^{-x},\omega)$.

\vskip .1in
\noindent 
Thanks to  the spherical symmetry of the potential $Q$, we can use separation of variables and the Fourier decomposition of $L^{2}(S^{d-1})$ to reduce (\ref{SchrodnewI}) to an infinite sequence of radial ordinary differential equations. In doing so, we shall denote by $\{Y_{k},\, k\geq 0\}$ an orthonormal Hilbert basis of $L^{2}(S^{d-1})$ consisting of eigenfunctions of $\triangle_{S}$, 
\begin{equation}\label{sphericalharm}
-\triangle_{S}Y_{k}=\alpha_{k}Y_{k}\,,\quad \alpha_{k}=k(k+d-2)\,.
\end{equation}
Note that we are committing a slight abuse of notation by omitting from the $Y_{k}$ the $d-2$ additional indices whose range reflects the multiplicities of the eigenvalues $\alpha_{k}$ of the Laplacian $-\triangle_{S}$ on $S^d$. In other words, our notation is a compact expression of the fact that any function $f \in L^{2}(S^{d-1})$ has a Fourier expansion in the $L^{2}$ sense given by 
\[
f=\sum_{k=0}^{\infty}f_k\,,
\]
where $f_k$ denotes the image of $f$ under the orthogonal projection from $L^{2}(S^{d-1})$ onto the restriction to  $S^{d-1}$ of the space of homogeneous harmonic polynomials of degree $k$ in $\mathbb{R}^d$.
This abuse of notation is carried further by writing the Fourier decomposition of v as
\begin{equation}\label{sepansatzI}
	v=\sum_{k=0}^{\infty} v_{k} (x) \  Y_{k}\,.
\end{equation}
giving rise to an infinite sequence of ordinary differential equations on  $(0,\infty)$ given by
\begin{equation}
	-v_{k}''+Qv_{k}=-(\alpha_k+\frac{(d-2)^2}{4})v_{k}=-\kappa_k^{2}v_{k}\,,
\end{equation}
where 
\[
\kappa_{k}:=k+\frac{d-2}{2}\,,\quad k\geq0\,.
\]

\vskip .1in
\noindent
As in \cite{DKN5},  we introduce the Weyl-Titchmarsh (WT) function $M(z)$ associated to the Sturm-Liouville operator $L$ on $(0, +\infty)$ given by
\begin{equation}\label{ComplexifiedSL}
	L=-\frac{d^2}{dx^2}+Q\,.
\end{equation}
This function will play a central role in our subsequent analysis of the stability problem of the Steklov spectrum for our Schr\"odinger operator \footnote{Since we are concerned with the Dirichlet-to-Neumann map, we are putting a Dirichlet boundary condition at $x=0$, which represents the boundary sphere ${\partial M}=S^{d-1}$ in the new radial coordinate $x=-\log r$. Had we put a Neumann condition at $x=0$, the corresponding Weyl-Titchmarsh function would have then corresponded to the Neumann-to-Dirichlet map, in which case the multiplication operators defined below in (\ref{multop}) would have involved the Weyl-Titchmarsh function associated to a Neumann boundary condition at $x=0$.}. We  assume in this paper that 
\begin{equation}\label{QL2I}
	Q\in L^{2}(0,\infty)\,.
\end{equation}
Under this assumption, it is well-known that $L$ is of limit point-type at infinity, which means that for all $z\in \mathbb{C}\setminus [-\beta, \infty)$ with $\beta >>1$, there exists, up to a zon-zero multiplicative constant, a unique solution $u(x,z)$ of 
\begin{equation}\label{ComplexifiedSLI}
	-u''+Qu=zu\, \quad z\in \mathbb{C}\,,
\end{equation} 
which is $L^{2}$ at $\infty$. The Weyl-Titchmarsh function $M(z)$ is then defined by
\begin{equation}\label{defWTI}
	M(z):=\frac{u'(0,z)}{u(0,z)}\, \ {\rm{for \ all  }} \ z\in \mathbb{C}\setminus [-\beta, \infty)\,.
\end{equation}

\noindent
Of course, the square-integrability hypothesis (\ref{QL2I}) we made on the potential $Q$ does not necessarily imply that the initial  potential $q\in L^{\infty}(M)$. Thus, the definition we gave for the DN map in (\ref{DNI}) is not directly applicable in this $L^2$ setting. We overcome this difficulty by exploiting the separation of variables and follow the procedure used in Section 2 of \cite{DKN5} to define the DN map in the present setting; namely we expand the boundary data $\psi$ in the Hilbert basis $\{Y_k,k\geq 0\}$ of $L^{2}(S^{d-1})$ as 
\[
\psi=\sum_{k=0}^{\infty}\psi_{k}Y_{k}\,,
\]
and define the DN map $\Lambda_{q}$  as a sum of operators $\Lambda_{q}^{k}$ by  
\begin{equation} 
	\Lambda_q \psi = \sum_{k =0}^\infty  (\Lambda_{q}^{k} \psi_k) Y_k\,.
\end{equation}
In other words, the operators $\Lambda^{k}_{q}$ are the restrictions of the Dirichlet-to-Neumann map $\Lambda_{q}$ to the eigenspaces corresponding to the eigenvalues $\sigma_k$ of $\Lambda_{q}$. Their expression is computed in the form of multiplication operators from the separation of variables for the Schr\"odinger operator and the Fourier decomposition of $L^{2}(S^{d-1})$ into orthogonal eigenspaces of the Laplacian $-\triangle_{S}$ on $S^{d-1}$:
\begin{equation} \label{DiagDNI}
	\Lambda^{k}_{q} \psi_k = -\frac{(d-2)}{2}v_{k}(0)-v'_{k}(0)\,,
\end{equation}
where 
\[
-v_{k}''+Qv_{k}=-\kappa_k^{2}v_{k}\,\quad v_{k}(0)=\psi_{k}, \,\  v_{k} \in L^2 \ \mbox{at} \, +\infty \,.
\]
As was proved in Section 2 of \cite{DKN5}, the operators $\Lambda^{k}_{q},\,k\geq 0$ can be further simplified by making use of the Weyl-Titchmarsh function $M$ evaluated at the points $-\kappa_k^2$,
\begin{equation}\label{multop}
	\Lambda^k_q \psi_k = \left( -\frac{(d-2) }{2} - M(-\kappa_k^2) \right) \psi_k\,,
\end{equation}
thus providing the expression of the Steklov spectrum $\{\sigma_{k},\, k\geq 0\}$ in terms of the Weyl-Titchmarsh function $M$ as  
\begin{equation}\label{StekSpecI}
	\sigma_k = -\frac{(d-2) }{4} - M(-\kappa_k^2)\,.
\end{equation}

\vskip .1in
\noindent
There is an important representation formula first obtained in \cite{Si1} for the Weyl-Titchmarsh function in terms of a the Laplace transform of a unique \emph{amplitude function} $A$, under the hypothesis that $Q\in L^{1}(0,\infty)$: 
\begin{equation}\label{SimonRepI} 
	M(-\kappa^2) = -\kappa - \int_0^\infty A(\alpha) e^{-2\kappa\alpha} d\alpha\,,\quad \forall \kappa > \frac{1}{2}||Q||_1\,.
\end{equation}
We shall use below a slightly refined version of this formula which applies in our $L^2$-setting and which will serve as the starting point of our formulation of the stability problem for the Steklov spectrum.

\vskip .1in
\noindent
Let us now explain what we mean precisely by {\it local stability estimates}. The starting point is a {\it fixed potential} $Q \in L^2 (0, \infty)$ which we perturb through the addition of a certain exponential series to its corresponding amplitude $A$. The set of admissible exponential series parametrizing the perturbations will be shown to lie in a certain infinite-dimensional space. Then using powerful results of Killip -Simon \cite{KS} we will show that the perturbed potential $\tilde{Q} \in L^2 (0,\infty)$. We then denote by $\{\tilde{\sigma}_k, k\geq 0\}$ the Steklov spectrum associated to ${\tilde{Q}}$. Since $ M(-\kappa_k^2) = - \kappa_k +o(1)$, as $k \to \infty$, (see \cite{Si1}, Corollary 4.2), we see that the sequence $\{\sigma_k -\tilde{\sigma}_k, k \geq 0\} \in \ell^{\infty}(\N)$. In the rest of the paper, we assume that the difference between their corresponding Steklov spectra is uniformly bounded in absolute value by a small error $\epsilon>0$; in other words, we set 
\begin{equation}\label{StabilityStatementI}
	||\sigma_k-{\tilde \sigma}_{k}||_{\ell^{\infty}(\N)}=:\epsilon .
\end{equation}
Our main goal is to estimate the  difference $Q-{\tilde{Q}}$ of the potentials. Again, our result is local: for any fixed parameter $T>0$,  we get  stability estimates in the space $L^2(0,T)$, meaning that
\begin{equation}
	|| \tilde{Q}- Q ||_{L^2 (0,T)} \leq C_T \ g(\epsilon)\,,
\end{equation}
where $g(\epsilon) \to 0$ when $\epsilon \to 0$, and $C_T$ is a constant depending only on $T$.

\vskip .2in
\noindent
Now, we can state our main result in this paper :

\begin{thm}\label{Mainest}
	Let $Q\in L^{2}(0,\infty)$ be a square-integrable potential with amplitude function $A$, and let   $\delta \geq 3-d$ be  any fixed parameter. Set $\mu_k := \lambda_k +\delta$ where $\lambda_k = 2k+d-3 +\delta$ and let $\{c_{k},\,k\geq 0\}$ be a sequence of real numbers such that 
	\begin{itemize}
		\item{i)} $c_k \leq 0$ for all $k \geq 0$.
		\item{ii)} The power series $\sum_{k\geq 0}c_{k} t^{\lambda_k}$ has a radius of convergence $R>1$.
	\end{itemize}
	Then the function ${\tilde{A}}$ defined by 
	
		\begin{equation}\label{Defperturbation}
			{\tilde{A}}(\alpha)=A(\alpha)+\sum_{k\geq 0}c_{k}e^{-\mu_{k}\alpha}\,,\quad \alpha >0\,,
		\end{equation}
	is the amplitude function of a potential ${\tilde{Q}}\in L^{2}(0,\infty)$. \par
	\vspace{0.2cm}\noindent
	Moreover, for any fixed $T>0$, there exists a positive constant $C_T$ such that
	
	\begin{equation}
		|| \tilde{Q}- Q ||_{L^2 (0,T)} \leq C_T  \ ||\sigma_k-{\tilde \sigma}_{k}||_{\ell^{\infty}(\N)}^{\theta}\,,
	\end{equation}
where the H\"older exponent $\theta \in (0, \half]$ is independent of $T$ and is given by 
	\begin{equation}
	\theta = \half  \min (1,  \frac{\log R}{\log (\frac{9M_0}{2})} ),\quad M_0=\max\{2,4(d-3+\delta)+1\}\,.
	\end{equation}
	
\end{thm}

\vskip .2in
\noindent
Note that when the radius of convergence $R \geq \frac{9M_0}{2}$, we can take as H\"older exponent $\theta= \half$. Note also that when the initial potential is the trivial potential $Q=0$, the perturbed potentials $\tilde{Q}$ can be seen as a generalization of the so-called Bargmann potentials (see section 5 for details). 

\vskip .1in
\noindent
Finally, it is noteworthy that with respect to the original Schr\"odinger operator, the type of perturbation being considered for the amplitude function $A$ amounts to the introduction of a finite number $N$ of negative eigenvalues $- \frac{\mu_k^2}{4}$ for $k=1, ..., N$, (corresponding to the case where $\mu_k$ is negative), and  of a countable set of  real resonances $- \frac{|\mu_k|}{2}$ which are equally spaced on the negative real axis (for $k$ greater than some $k_0$). These resonances are quantified explicitly in terms of the parameter $\delta$ and the eigenvalues of the Laplace Beltrami operator $\Delta_S$ on the boundary sphere.

\vspace{0.5cm}
\noindent As a byproduct, we also obtain local H{\"o}lder stability estimates for the Calder\'on problem for radial Schr\"odinger operators on the unit ball. We recall that the initial  potential $q(r)$ is related to $Q(x)$ thanks to the relation  $Q(x) = e^{-2x}q(e^x)$. In particular, we easily see that 
$Q \in L^2((0,\infty), dx)$ if and only if $q \in L^2((0, 1), r^3 dr)$, (resp. $Q \in L^2((0,T), dx)$ if and only if $ q \in L^2((e^{-T}, 1), r^3 dr)$). One easily gets from Theorem \ref{Mainest} :

\vskip .2in

\begin{coro}\label{Calderon1}
Let $q\in L^{2}((0,1), r^3 dr)$ be a fixed central potential and let $\tilde{q}$ be the potential associated with $\tilde{Q}$ given in Theorem \ref{Mainest}. Then,  $\Lambda_q - \Lambda_{\tilde{q}}$ is a bounded operator on $L^2(S^{d-1})$,  and for any fixed $T>0$, there exists a positive constant $C_T$ such that
	\begin{equation}
		|| q- \tilde{q}||_{L^2 ((e^{-T},1), r^3) dr} \leq C_T  \ || \Lambda_q - \Lambda_{\tilde{q}} ||_{B(L^2(S^{d-1}))}^{\theta}\,,
	\end{equation}
	where  $\theta \in (0, \half]$ is the H\"older exponent given in Theorem \ref{Mainest}.
\end{coro}

\vskip .1in

\begin{rem}
It is important to  recall that, generically, the problem of the determination of the potential $q$ from the DN map $\Lambda_q$ is highly unstable as shown by Mandache in \cite{Man}. More precisely, in \cite{Man}, Theorem 1 and Corollary 2, Mandache constructs potentials $q$ supported in the unit ball, around which the inverse problem is exponentially unstable. In particular, one cannot get H\"older stability estimates for such potentials. These potentials are not necessarily radial but Mandache states that even radial potentials give counterexamples to stability (see the remark before Lemma 4). We emphasize there is no contradiction with the H\"older stability estimates obtained in Corollary \ref{Calderon1} below. Indeed, the potentials $\tilde{Q}$ appearing in Theorem \ref{Mainest} are very special since they are connected to the particular 
amplitude $\tilde{A}(\alpha)$ given in (\ref{Defperturbation}).
\end{rem}


\section{Notation and set-up of the model}\label{setup}
\noindent On the $d$-dimensional closed Euclidean ball
\begin{equation} \label{manifold}
  M = (0,1] \times S^{d-1}\,,
\end{equation} 	
where $d \geq 3$, we consider the Dirichlet problem for the Schr\"odinger operator with potential $q\in L^{\infty}(M)$, given by
\begin{equation} \label{Schrod}
  \left\{ \begin{array}{cc}
	-\triangle u + q\,u= 0, & \textrm{on} \ M\,, \\
	u = \psi \in H^{\half}(\partial M), & \textrm{on} \ \partial M\,.
	\end{array} \right.
\end{equation}
\vskip .1in
\noindent Although we shall shortly make several assumptions about the potential $q$ (including the requirement that it be radial), we begin by recalling a few general facts concerning the Dirichlet problem (\ref{Schrod}) for a general potential $q\in L^{\infty}(M)$. These will lead us to the definition of the Dirichlet-to-Neumann (DN) map and the associated Steklov spectrum, which will be central objects of study in this paper. 
\vskip .1in
\noindent We first recall (see for example Theorem 8.3 in \cite{Gilbarg}) that if $q\in L^{\infty}(M)$ and zero is not a Dirichlet eigenvalue of the above Schr\"odinger operator (which is the case for example if $q\geq 0$), then the Dirichlet problem (\ref{Schrod}) has a unique solution $u\in H^{1}(M)$. 
The Dirichlet-to-Neumann (DN) map $\Lambda_q: H^{1/2}(\partial M) \to H^{-1/2}(\partial M)$ is then defined by
\begin{equation} \label{DN}
  \Lambda_{q} \psi = \left( \partial_\nu u \right)_{|\partial M},
\end{equation}
where $u$ is the unique solution of (\ref{Schrod}) and $\left( \partial_\nu u \right)_{|\partial M}$ is the normal derivative of $u$ with respect to the outer unit normal vector
$\nu$ on $\partial M$. Here $\left( \partial_\nu u \right)_{|\partial M}$ is generally defined in the weak sense as an element of $H^{-1/2}(\partial M)$ by

{$$
  \left\langle \Lambda_{q} \psi | \phi \right \rangle = \int_M (\langle du, dv \rangle + quv) \, dVol\,,
$$
for any $\psi \in H^{1/2}(\partial M)$ and $\phi \in H^{1/2}(\partial M)$ such that $u$ is the unique solution of (\ref{Schrod}) and $v$ is any element of $H^1(M)$ such
that $v_{|\partial M} = \phi$. It is easily checked that if $\psi$ is sufficiently smooth, we have
$$
  \Lambda_{q} \psi = g(\nu, \nabla u)_{|\partial M} = du(\nu)_{|\partial M} = \nu(u)_{|\partial M}\,,
$$
so that the expression in local coordinates for the normal derivative is then given by
\begin{equation} \label{DN-Coord}
\partial_\nu u = \nu^i \partial_i u\,.
\end{equation}
\vskip .1in
\noindent It is well known that the DN map is a pseudo-differential operator of order $1$ which is self-adjoint on $L^2(\partial{M}, dS_g)$ where $dS_g$ denotes the metric induced by the Euclidean metric on the boundary $\partial M=S^{d-1}$.
Therefore, the DN map  $\Lambda_{q}$ has a real and discrete spectrum accumulating at infinity, known as the Steklov spectrum. We shall denote the Steklov eigenvalues (counted with multiplicity) by
\begin{equation} \label{Steklov}
  0 = \sigma_0 < \sigma_1 \leq \sigma_2 \leq \dots \leq \sigma_k \to \infty\,.
\end{equation}
We refer the reader to \cite{GP} and references therein for an excellent survey of the known results on the Steklov spectrum.
\vskip .1in
\noindent As of now and for the remainder of this paper, we shall assume that the potential is radial and write $q=q(r)$, where $r$ denotes the Euclidean distance to the origin. It will be convenient for our subsequent analysis to replace the radial coordinate $r\in (0,1]$ by a new radial coordinate $x\in [0,\infty)$ defined by $x=-\log r$, in which case the the boundary of $M$ now corresponds to $x=0$. The Euclidean metric 
\[
g=dr^2+r^2d\Omega^2\,,
\]
where $d\Omega^2$ denotes the round metric on the unit sphere $S^{d-1}$, then takes the form 
\[
g=f(x)^{4}(dx^2+d\Omega^2)\,,
\]
where $f(x)= \exp(-x/2)$, and the Dirichlet problem (\ref{Schrod}) gets transformed into
\begin{equation} \label{Schrodnew}
  \left\{ \begin{array}{cc}
	  [-\partial_x^2 - \triangle_{S} + Q(x)] v = -\frac{(d-2)^2}{4} v, & \textrm{on} \ M\,, \\
	v = f^{d-2} \psi, & \textrm{on} \ \partial M\,,
	\end{array} \right.
\end{equation}
where $\triangle_{S}$ denotes the Laplacian on the boundary sphere $S^{d-1}$, $v=f^{d-2}u$ and $Q(x):=e^{-2x}q(e^{-x})$.

\subsection{The Weyl-Titchmarsh function}
We now return to the Weyl-Titchmarsh function $M(z)$ which we introduced in (\ref{defWTI}) for the Sturm-Liouville operator 
\begin{equation}
L=-\frac{d^2}{dx^2}+Q\,,
\end{equation}
defined by the left-hand side of (\ref{ComplexifiedSL}). This function will play a central role in our subsequent analysis of the stability problem of the Steklov spectrum for our Schr\"odinger operator. 
\vskip .1in
\noindent We first recall that in order for the Weyl-Titchmarsh function to be well-defined, we need to assume that $L$ is of limit point-type at infinity, meaning that for all $z\in \mathbb{C}\setminus [-\beta, \infty)$ with $\beta >>1$, there exists, up to a zon-zero multiplicative constant, a unique solution $u(x,z)$ of (\ref{ComplexifiedSL}) which is $L^{2}$ at $\infty$. The Weyl-Titchmarsh function $M(z)$ is then defined by
\begin{equation}\label{defWT}
M(z):=\frac{u'(0,z)}{u(0,z)}\, \ {\rm{for \ all  }} \ z\in \mathbb{C}\setminus [-\beta, \infty)\,.
\end{equation}
It is easy to show that the property of $L$ being of limit point-type at infinity will be guaranteed if 
\begin{equation}\label{QL2}
Q\in L^{2}(0,\infty)\,,
\end{equation}
an assumption that we will require $Q$ to satisfy from now onwards.
Indeed, from (1.5) in \cite{Si1}, we know that $L$ will be of limit point-type at infinity if 
\begin{equation}\label{Beta2}
\beta_{2}:=\sup_{y>0}\int_{y}^{y+1}\max\{Q(x),0\}dx<\infty\,.
\end{equation}
But by the Cauchy-Schwarz inequality, we have 
\begin{equation}\label{Beta2Est}
\sup_{y>0}\int_{y}^{y+1}\max\{Q(x),0\}dx \leq \sup_{y>0}\int_{y}^{y+1}|Q(x)|dx\leq ||Q||_{2}\,,
\end{equation}
which shows that (\ref{QL2}) ensures indeed the property that $L$ is of limit point-type at infinity,

\subsection{The amplitude $A$ of a radial potential and the Weyl-Titchmarsh function}
As stated in the Introduction, central to our analysis of the stability problem lies a remarkable representation formula first obtained in \cite{Si1} (under the hypothesis that $Q\in L^{1}(0,\infty)$) for the Weyl-Titchmarsh function the Sturm-Liouville operator $L$ in terms of a the Laplace transform of an \emph{amplitude function} $A$. We shall be using a slightly refined version of this formula which applies to the class of $L^2$ potentials considered in our paper. 
\vskip .1in
\noindent In order to state this formula, we first recall from Theorem 2.1 in \cite{Si1} that if $Q\in L^{1}(0,\infty)$, then the Weyl-Titchmarsh function $M$ may be expressed in the form of the Laplace transform of an amplitude function $A$ by   
\begin{equation}\label{SimonRep} 
  M(-\kappa^2) = -\kappa - \int_0^\infty A(\alpha) e^{-2\kappa\alpha} d\alpha\,,\quad \forall \kappa > \frac{1}{2}||Q||_1\,.
\end{equation}
It was proved in \cite{GS1} (in the remark following (1.17)) that the above equality also holds for all $\kappa \in \mathbb{C}$ such that ${\mbox{Re}}\,\kappa > \frac{1}{2}||Q||_1$. In \cite{AMR} (see Section 5, Algorithm 1, point 2), it is proved that if $\beta_{2}<\infty$, where $\beta_2$ is defined in 
(\ref{Beta2}), then the integral in (\ref{SimonRep}) is absolutely convergent for ${\mbox{Re}}\,\kappa >2\max\{\sqrt{2\beta_2}, e \beta_2\}$. But we saw in (\ref{Beta2Est}) that for square-integrable potentials $Q$, one has the estimate $\beta_{2}\leq ||Q||_{2}<\infty$. It therefore follows from (\ref{SimonRep}) that the Weyl-Titchmarsh function $M$ admits the representation
\begin{equation}\label{SimonRepfinal}
 M(-\kappa^2) = -\kappa - \int_0^\infty A(\alpha) e^{-2\kappa\alpha} d\alpha\,,\quad \forall\, {\mbox{Re}}\, \kappa > 2\max\{\sqrt{2\beta_2}, e\beta_2\}\,.
\end{equation}

\Section{The problem of stability}
\subsection{Statement of the problem and strategy}
The stability problem may be stated as follows in general terms: Given a pair of potentials $Q,{\tilde{Q}}$ such that the difference between their corresponding Steklov spectra is uniformly bounded in absolute value by a small error $\epsilon>0$, i.e

\begin{equation}\label{StabilityStatement}
	||\sigma_k-{\tilde \sigma}_{k}||_{\ell^{\infty}(\N)}=:\epsilon, 
\end{equation} 
what can we say about the difference $Q-{\tilde{Q}}$ of these potentials? 
\noindent As a first step, we can use the expression (\ref{StekSpecI}) of the Steklov spectrum in terms of the Weyl-Titchmarsh function $M$ and Simon's representation formula (\ref{SimonRepfinal}) for $M$ in terms of the Laplace transform of the amplitude function $A$ to reformulate the condition (\ref{StabilityStatement}) in terms of $A$. We have 
\[
|\sigma_k-{\tilde \sigma}_{k}|=|M(-\kappa_k^2)-{\tilde{M}}(-\kappa_k^2)|=|\int_0^\infty (A(\alpha)-{\tilde {A}}(\alpha)) e^{-2\kappa_{k}\alpha} d\alpha |\,,
\]
and making the change of variables $\alpha = -\log t$, our hypothesis (\ref{StabilityStatement}) on the difference of the Steklov spectra takes the form 
\begin{equation}\label{AEstimation}
|\int_{0}^{1} t^{-\delta} \big(A(-\log t)-{\tilde {A}}(-\log t)\big)\,t^{2k+d-3 +\delta}dt| \leq \epsilon\, ,
\end{equation}
where $\delta$ will be a fixed real parameter that will be properly chosen later.  We can see that (\ref{AEstimation}) is effectively a Hausdorff moment problem, and thus one should not expect better stability results than the logarithmic stability estimates of the type obtained in \cite{DKN5} for the Steklov spectra of deformed balls, or in \cite{Ale} and \cite{Nov} for the Steklov spectra of certain Schr\"odinger operators. Nevertheless, as we shall explain in Section \ref{ImprovStab} below, one can approach the stability problem from a different starting point by working directly with perturbations of Simon's amplitude function $A$ by a certain families of exponential series obtained from power series of M\"untz type. We shall see that his leads in turn to H\"older type stability results which are significantly stronger than the logarithmic stability results mentioned earlier, albeit at the cost of restricting the class of potentials to a somewhat small subset of the set of square integrable potentials. More precisely, given the amplitude function $A$ associated to a square-integrable potential function $Q$ by (\ref{SimonRepfinal}), our strategy will consist in defining a perturbed amplitude function $\tilde{A}$ as in (\ref{APert}) and then use an important stability estimate due to Still (Theorem 2 in \cite{St}) to obtain a H\"older estimate on ${\tilde{A}}-A$. This will be the substance of Section \ref{ImprovStab}. The next step, worked out in Section \ref{KS}, will consist in using the powerful results of \cite{KS} to construct an $L^2$ potential $\tilde{Q}$ associated to the perturbed amplitude $\tilde{A}$. Finally, we shall use the methods of boundary control theory of \cite{AMR} to estimate the difference ${\tilde{Q}}-Q$. \\
\par 
\noindent It is noteworthy that with respect to the original Schr\"odinger operator, the type of perturbation being considered for the amplitude function $A$ amounts to the introduction of a finite number of negative eigenvalues (corresponding to the choice of a negative real parameter $\delta$) and  of a countable set of resonances on the negative  real axis, which admit a precise quantitative expression through to the eigenvalues $\mu_k$ of the Laplace Beltrami operator $\Delta_S$ on the boundary sphere.

\subsection{Improved stability by Still's method - first main result}\label{ImprovStab}
Let $Q\in L^{2}(0,\infty)$ be a given fixed potential and let $A$ be the corresponding amplitude function, given by the representation formula (\ref{SimonRep}). 
\vskip .1in
\noindent Having in mind the inequality (\ref{AEstimation}), we set for $k \geq 0$,
\begin{equation}\label{lambdak}
\lambda_{k}:=2k+d-3 +\delta\,,
\end{equation}
where $\delta \geq 3-d$ is an arbitrary fixed real parameter (so that $\lambda_k \geq 0$) and
\begin{equation}\label{defh} 
	h(t)=  t^{-\delta} \left( \tilde{A}(-\log t)-A(-\log t) \right)
\end{equation}
We  define  formally a new amplitude ${\tilde{A}}$ by adding to $A$ a power series 
\begin{equation}\label{APert}
{\tilde{A}}(\alpha)=A(\alpha)+\sum_{k\geq 0}c_{k}e^{-(\lambda_{k}+\delta)\alpha}\,,
\end{equation}
or equivalently
\begin{equation}
h(t) =\sum_{k\geq 0}c_{k}t^{\lambda_{k}}\,.
\end{equation}
  \\ \par 
We assume that the series defining $h(t)$  has a radius of convergence $R>1$, so that $h \in C^0 ([0,1])$. Furthermore we assume that $h$ is such that the estimate (\ref{AEstimation}) holds, that is 
\begin{equation}
 |\int_{0}^1h(t)\,t^{\lambda_{k}}dt|\leq \epsilon \,,\quad \forall k\geq 0\,.
\end{equation}

Our goal for this section is to obtain a good approximation of H\"older or Lipschitz type for $||h||_{{2}}^{2}$, under the above assumptions. We shall do so by using Theorem 2 in the paper \cite{St} by Still and the procedure used in Section 4.3 of \cite{DKN5}. In order to do so we first recall some of the notation used in \cite{DKN5}.
\begin{thm}
Given $\epsilon>0$ and $R>1$ as above and letting $M_0=\max\{2,4(d-3 +\delta)+1\}$, we have, for some universal constant $B>0$, the estimate 
\begin{equation}\label{l2estimateh}
||h||_{2}^{2}\leq B^{2}\epsilon+R^{1-d}\e^{\frac{\log R}{\log (\frac{9M_0}{2})}}\,.
\end{equation}
\end{thm}
We note that the estimate (\ref{l2estimateh}) is generally a H\"older type estimate for $||h||_2^{2}$, but that if $R>\frac{9M_0}{2}$, this estimate is Lipschitz. 
\vskip .1in 
\begin{proof} \noindent Given a sequence $\Lambda_{\infty}:=(\lambda_{n})_{n\geq 0}$ of integers such that $0\leq \lambda_{0}<\lambda_{1}<\cdots $ and $\lambda_{k}\to \infty$ as $k\to \infty$, we define for fixed $ n\geq 1$ the finite sequence
\begin{equation}
 \Lambda_n := 0 \leq \lambda_0<\lambda_1 < ... <\lambda_n\,,
\end{equation}
giving rise to the vector space $ \mathcal{M}(\Lambda)$ of "M\"untz polynomials of degree $\lambda_n$":
\begin{equation}
 \mathcal{M}(\Lambda_{n}) = \{ P : \ P(t) = \sum_{k=0}^n a_k \ t^{\lambda_k} \}\,.
\end{equation}
Recall that according to the M\"untz-Sz\' asz Theorem, if $\Lambda_{\infty}$ is a sequence of positive real numbers as above, then span $\{ t^{\lambda_0}, t^{\lambda_1}, ... \}$  is dense in $L^2([0,1])$ if and only if
\begin{equation}\label{condM\"untz}
\sum_{k=1}^{\infty} \frac {1}{\lambda_k} = \infty.
\end{equation}
\vskip .1 in 
\noindent We remark that if $\lambda_0=0$, the denseness of the M\"untz polynomials in $C^0([0,1])$ in the sup norm is also characterized by (\ref{condM\"untz}).
\vskip .1in
\noindent Given now a function $f$ in $C^0([0,1])$ or in $ L^2([0,1])$, the error of approximation of $f$ with respect to $\mathcal{M}(\Lambda_{n})$ is defined by
\begin{equation}
 E_{k}(f, \Lambda_{n}):= \inf_{P \in \mathcal{M}(\Lambda_{n})} \ || f-P||_k \,,
\end{equation}
where $k=2$ or $k=\infty$ depending on whether $f \in C^0([0,1])$ or $ f \in L^2([0,1])$. For our application, we have $\lambda_{k}:=2k+d-3+\delta$, giving $\lambda_{k+1}-\lambda_{k}=2>0$, so by Theorem 2 of \cite{St}, we know that 
\begin{equation}\label{StillEstimate}
E_{\infty}(h,\Lambda_{n})\leq CR^{-\lambda_{n+1}}\,,
\end{equation}
for some positive constant $C$.
\vskip .1in
\noindent We have, denoting by $\pi_n$ the orthogonal projection onto the subspace $\mathcal{M}(\Lambda_{n})$,
\begin{equation}
||h||_{2}^{2}=||\pi_{n}(h)||_{2}^{2}+||h-\pi_{n}(h)||_{2}^{2}\,.
\end{equation}
\vskip .1in
\noindent Our next step is to combine the estimate (4.57) from \cite{DKN5} Section 4.3 of and the estimate (\ref{StillEstimate}) to obtain an estimate for the norm of $\pi_{n}(h)$. In order to do so, we use the Gram-Schmidt process to obtain polynomials $(L_m(t))$ with $L_0(t)=1$, and for $m \geq 1$,
\begin{equation}\label{Lm}
L_m(t) = \sum_{j=0}^m C_{mj} t^{\lambda_j},
\end{equation}
where we have set
\begin{equation}\label{cm}
C_{mj}= \sqrt{2\lambda_m +1} \  \frac{ \prod_{r=0}^{m-1} (\lambda_j + \lambda_r +1)}{ \prod_{r=0, r \not= j}^{m} (\lambda_j - \lambda_r)}.
\end{equation}
The family $(L_m (t))$ defines an orthonormal Hilbert basis of $L^2([0,1])$. 
We may now recall the estimate (4.57) from \cite{DKN5}, 
\begin{equation}
||\pi_{n}(h)||_{2}^{2}\leq \e^2 \ \sum_{k=0}^n \left( \sum_{p=0}^k |C_{kp}| \right)^2\,,
\end{equation}
which gives immediately
\begin{equation}\label{l2estimate}
||h||_{2}^{2}\leq \e^2 \ \sum_{k=0}^n \left( \sum_{p=0}^k |C_{kp}| \right)^2 + CR^{-\lambda_{n+1}}\,,
\end{equation}
using (\ref{StillEstimate}) and the inequality
\[
||h-\pi_{n}(h)||_{2}^{2}=E_{2}(h,\Lambda_n)\leq E_{\infty}(h,\Lambda_n)\,.
\]
Now, according to the estimate (4.72) of \cite{DKN5}, we have 
\begin{equation}\label{estimproj}
  ||\pi_{n} h||_2^2 \leq B^2 \e^2 \ g(n)^2\,,
\end{equation}
where $B$ is a positive constant and $ g :[0, + \infty[$ is a monotone increasing function defined for $t \in [0,+\infty)$ by
\begin{equation}\label{Defg}
 g(t) = \frac{3}{2}   \frac {  1   } {\sqrt{  \left(\frac{9M_0}{2}\right)^{2}-1}} \ \sqrt{2t+1} \ \left(\frac{9M_0}{2}\right)^{t+1}\quad , \quad M_0=\max\{2,4(d-3+\delta)+1\}\,.
\end{equation}
\vskip .1in
\noindent
Now, repeating the steps that lead from the inequalities (4.72) to (4.73) in \cite{DKN5}, we choose $n$ as a function of $\epsilon $ so as to control the norm of the projection $||\pi_{n} h||_2^2$ of $h$ and thus set ${\displaystyle{n(\epsilon) := [ \ (g^{-1} ( \frac{1}{\sqrt{\e}}))]}}$ where square brackets denote the integral part function. Since $g$ is a monotone increasing function, we have 
\begin{equation}\label{gnEstimate}
g(n(\e)) \leq \frac{1}{\sqrt{\e}}\,,
\end{equation}
so using (\ref{estimproj}) we obtain immediately:
\begin{equation}\label{stability1}
 ||\pi_{n(\e)} h||_2^2 \leq B^2 \e.
\end{equation}
\vskip .1in
\noindent Our next task is now to estimate the size of $n(\epsilon)$ relative to $\epsilon$ so as to obtain the H\"older estimate we seek for $||h||_{2}^{2}$. From (\ref{Defg}), we obtain that 
\[
g(t)\sim (t+1)\log(\frac{9M_0}{2})\,,
\]
as $t\to \infty$, which combined with (\ref{gnEstimate}) leads to 
\begin{equation}\label{nepsilon}
n(\epsilon)=\frac{\log(\frac{1}{\sqrt{\epsilon}})}{\log(\frac{9M_0}{2})}\,.
\end{equation}
Plugging this into (\ref{l2estimate}) gives
\begin{equation}\label{l2estimate1}
||h||_{2}^{2}\leq B^{2}\epsilon+ CR^{-\lambda_{n(\epsilon)+1}}\,.
\end{equation}

Now, using the expression $\lambda_{k}=2k+d-3+\delta$ and (\ref{nepsilon}), we have
\begin{equation}\label{preestimate}
R^{-\lambda_{n(\epsilon)+1}}=R^{1-d-\delta}R^{-2n(\epsilon)}\sim R^{1-d-\delta}R^{-\frac{2 \log \frac{1}{\sqrt{\e}}}{\log (\frac{9M_0}{2})}}\sim  R^{1-d-\delta}e^{\frac{\log \e}{\log(\frac{9M_0}{2})}\log R}\sim R^{1-d-\delta}\e^{\frac{\log R}{\log( \frac{9M_0}{2})}}\,.
\end{equation}
Substituting (\ref{preestimate}) into (\ref{l2estimate1}), we obtain
\begin{equation}\label{l2estimate2}
||h||_{2}^{2}\leq B^{2}\epsilon+R^{1-d-\delta}\e^{\frac{\log R}{\log( \frac{9M_0}{2})}}\,.
\end{equation}
\vskip .1in
\noindent In terms of the amplitude function $A$ in the variable $\alpha \in (0,\infty)$, using the relation  
\[
||h||_2^{2}=\int_{0}^{1} t^{-2\delta} ({A}(-\log t)-{\tilde{A}}(-\log t))^{2}\,dt\,,
\]
we obtain 
\begin{equation}\label{l2estimateA}
\int_{0}^{\infty}e^{(2\delta-1)\alpha}({A}(\alpha)-{\tilde{A}}(\alpha))^{2}\ d\alpha \leq B^{2}\epsilon+R^{1-d-\delta}\e^{\frac{\log R}{\log (\frac{9M_0}{2})}}\,.
\end{equation}

\end{proof}

\Section{From the perturbed amplitude $\tilde A$ to a potential ${\tilde {Q}} \in L^{2}(0,\infty)$ }\label{KS}

\subsection{Statement of the second main result}
Our objective for Section \ref{KS} is to establish a result on the existence of square-integrable potentials ${\tilde {Q}}$ associated to the perturbed amplitudes $\tilde A$ as defined in (\ref{APert}). As we shall see, this will require a few additional hypotheses on the perturbation of the amplitude function $A$ given by (\ref{APert}) and thus on the perturbation of the starting potential $Q$. 

\vskip .1in
\noindent

We set for $k \geq 0$,
\begin{equation}\label{muk}
	\mu_k := \lambda_k +\delta = 2k+d-3 +2\delta
\end{equation}
so that
\begin{equation}\label{defperturbation0}
	{\tilde{A}}(\alpha)=A(\alpha)+\sum_{k\geq 0}c_{k}e^{-\mu_k \alpha}\,,\quad \alpha >0\,.
\end{equation}

\vskip .1in
\noindent 
Thus, in dimension $d$ greater than 3, $\mu_k$ may be a negative real number, so we split the series in  (\ref{defperturbation0}) as
\begin{equation}\label{defperturbation1}
	{\tilde{A}}(\alpha)=A(\alpha)+\sum_{k= 0}^{N-1} c_{k}e^{-\mu_k \alpha} + \sum_{k\geq N}c_{k}e^{-\mu_k \alpha} \,,\quad \alpha >0\,,
\end{equation}
such that for $k=0, ..., N-1, \ \mu_k < 0$ and for $k \geq N+1, \ \mu_k \geq 0$, with the convention that the first sum in (\ref{defperturbation1}) does not appear if all the $\mu_k$'s are positive (i.e if $N=0$).

\vskip .1in
\noindent 
Now, it is convenient to  rewrite  (\ref{defperturbation1}) as:
\begin{equation}\label{defperturbation2}
	{\tilde{A}}(\alpha)-A(\alpha) = \sum_{k= 0}^{N-1} 2c_{k}\ \sinh (|\mu_k|\alpha) + \sum_{k\geq 0}c_{k}e^{-|\mu_k|\ \alpha}.
\end{equation}

\vskip .1in
\noindent 
In this  form, we note that the perturbation ${\tilde{A}}(\alpha)-A(\alpha)$ has exactly the same expression as the amplitude given in \cite{GS1}, Eq. (11.9), modulo the fact we consider a convergent series instead a finite sum, and that the coefficients $|\mu_k|$ for $k=0, ..., N-1$ appear simultaneously in the first finite sum and also in the convergent series.


\vskip .1in
\noindent 
As we shall see in the next section, the first finite sum in the right-hand side of (\ref{defperturbation2}) corresponds to the introduction of negative eigenvalues {$- \frac{\mu_k^2}{4}$} for $k=0, ..., N-1$,  whereas the second sum corresponds to the introduction of  real resonances $- \frac{|\mu_k|}{2}$ for $k \geq 0$.


\vskip .1in
\noindent 
We state this result in the form of a theorem, that will be proved in Section \ref{KillipSimon} below by applying Theorem 1.2 from the paper \cite{KS} by Killip and Simon:
\begin{thm}\label{KSApp}
Let $Q\in L^{2}(0,\infty)$ be a square-integrable potential with amplitude function $A$, let $\{c_{k},\,k\geq 0\}$ be a sequence of real numbers such that 
\begin{itemize}
\item{i)} For all $k \geq 0,\  c_k \leq 0$.
\item{ii)} the power series $\sum_{k\geq 0}c_{k} t^{\lambda_k}$ has a radius of convergence $R>1$.
\end{itemize}
Then the function ${\tilde{A}}$ defined by 

\begin{equation}\label{defperturbation}
{\tilde{A}}(\alpha)=A(\alpha)+\sum_{k\geq 0}c_{k}e^{-\mu_{k}\alpha}\,,\quad \alpha >0\,,
\end{equation}
is the amplitude function of a potential ${\tilde{Q}}\in L^{2}(0,\infty)$.
\end{thm}

\subsection{The spectral measure}
In order to proceed with the proof of Theorem \ref{KSApp}, we need to first compute the difference of the spectral measures of $Q$ and  ${\tilde {Q}}$ in terms of the data contained in the perturbed amplitude ${\tilde {A}}$. In the first instance we will work heuristically so as to set the stage for the mathematical objects at play.
\vskip .1in
\noindent Let us first recall from \cite{GS2} that the Weyl-Titchmarsh function $M$ is a function of Herglotz type, meaning that for all 
$z\in \mathbb{C}$ such that ${\mbox{Im}}\,z>0$ we have ${\mbox{Im}}\,M(z)>0$. The Weyl-Titchmarsh function belongs to the subclass of functions of Herglotz type admitting the representation formula

\[
M(z)=c+ dz+ \int_{\mathbb{R}}\big(\frac{1}{\lambda-z}- \frac{1}{1+\lambda^{2}}\big)d\rho(\lambda)\,,
\]
where $c= Re (M(i))$,  $d = \lim_{y \to + \infty} \frac{M(iy)}{y} = 0$, (see \cite{Si1}, Corollary 4.2), and  $d\rho(\lambda)$ is the (positive) spectral measure associated to (\ref{ComplexifiedSL}). The measure $d\rho(\lambda)$ can be constructed by taking the following weak limit (in the distributional sense)
\begin{equation}\label{weaklim}
d\rho(E)=\plim[w]_{\epsilon \downarrow 0}\frac{1}{\pi}{\mbox{Im}\,(M}(E+i\epsilon))\,dE\,.
\end{equation}
We denote by $\tilde{M}(-\kappa^2)$ the putative Weyl-Titchmarsh function associated with the amplitude $\tilde{A}(\alpha)$, so thanks to (\ref{SimonRepfinal}), we have  (formally and for suitable $\kappa$),

\begin{equation}\label{DiffWT}
{\tilde{M}}(-\kappa^2) - M(-\kappa^2)= - \int_{0}^{\infty} \left( {\tilde{A}}(\alpha)-A(\alpha) \right) \ e^{-2\kappa \alpha} \ d\alpha .
\end{equation}

\vskip .1in
\noindent
Now, using (\ref{defperturbation2}), we easily write  the difference $\tilde{M}(-\kappa^2) - M(-\kappa^2)$ as 
\begin{equation}
{\tilde{M}}(-\kappa^2) - M(-\kappa^2)= - 2 \sum_{k= 0}^{N-1} c_{k}\  \frac{|\mu_k|}{4\kappa^2 -\mu_k^2} \ -\ \sum_{k\geq 0}\frac{c_k}{2\kappa +|\mu_{k}|}\,.
\end{equation}

\noindent
We note that the above series is indeed convergent since by hypothesis the power series $\sum_{k\geq 0}c_{k}t^{\lambda_k}$ 
has radius of convergence $R>1$ so that in particular $\sum_{k\geq 0}|c_{k}|<\infty$.

\vskip .1in
\noindent
Now, using (\ref{weaklim}), we are able to define the difference of the spectral measures $d\tilde{\rho}(E) - d\rho (E)$ (for more details, we refer the reader to (\cite{GS1}, Eqs. (11.7)-(11.9), p.637).

\vskip .1in
\noindent
For $E \geq 0$,
\begin{equation}\label{Diffmeas1}
	d\tilde{\rho}(E) = d\rho (E) -\frac{2}{\pi} \sum_{k \geq 0}  c_k \ \frac{\sqrt{E}}{4E+\mu_k^2} \ dE,
\end{equation}
and for $E<0$,
\begin{equation}\label{Diffmeas2}
	d\tilde{\rho}(E) = d\rho (E) - \half \sum_{k=0}^{N-1} c_k\ |\mu_k| \  \delta (\cdot + \frac{\mu_k^2}{4} ) \ dE ,
\end{equation}
where $\delta(\cdot- a)$ stands for the usual delta distribution, centered at the point $a$.

\vskip .1in
\noindent
\vskip .1in
\noindent
As was explained in (\cite{GS1}, Section 11),  this corresponds to the introduction of a finite number of negative eigenvalues $-\frac{\mu_k^2}{4}$ for $k=0, ..., N-1$, and  of real resonances $-\frac{|\mu_k|}{2}$ for $k \geq 0$.

 \vskip .1in
 \noindent
 Now, we can explain precisely our strategy : in the next section we show, using the Killip-Simon conditions, that under the hypotheses of Theorem \ref{KSApp}, 
 there exists a potential ${\tilde{Q}}\in L^{2}(0,\infty)$ 
 associated to the above spectral measure $d\tilde{\rho}(E)$, allowing us to define the associated Weyl-Titchmarsh function 
 $\tilde{M} (z)$ for $z \in \C \backslash [-\tilde{\beta}, +\infty[$ for $\tilde{\beta} >>1$. The amplitude function associated to $\tilde{M} (z)$ is automatically given by $\tilde{A}(\alpha)$ thanks to  
 the uniqueness of the inverse Laplace transform and analytic continuation.

 \subsection{The Killip-Simon conditions and the proof of Theorem \ref{KSApp}}\label{KillipSimon}
As earlier, we now apply Theorem 1.2 from the paper \cite{KS} by Killip and Simon in order to prove Theorem \ref{KSApp} establishing existence of a potential ${\tilde{Q}}\in L^{2}(0,\infty)$ associated to the perturbed amplitude $\tilde A$. Besides the positivity of the perturbed measure $d{\tilde{\rho}}$, there are four conditions stated in the theorem of Killip and Simon that we need to verify on $d{\tilde{\rho}}$ in order for their theorem to apply. In what follows, we state these conditions and show they are satisfied under the hypotheses of Theorem \ref{KSApp}.
\begin{itemize}
\item {Positivity of the measure $d{\tilde{\rho}}(E)$:}

\noindent This follows immediately on account of (\ref{Diffmeas1}), (\ref{Diffmeas2}), and the hypothesis $c_{k}\leq 0\,, \forall\, k\geq 0$.

\item{Weyl condition:}
\noindent The Weyl condition on $d{\tilde{\rho}}$ states that the support of $d{\tilde{\rho}}$ should decompose as 
\[
\supp d{\tilde{\rho}} =[0,\infty)\cup \{{\tilde{E}}_{j}\}_{j=1}^{N}\,, \quad \mbox{with}\,\, {\tilde{E}}_{1}<{\tilde{E}}_{2}<\cdots<0\,, \quad \mbox{with}\,\, {\tilde{E}}_{j} \to 0 \,\, \mbox{if}\, N=\infty\,.
\]
Again this is immediate from the fact that the measure $d\rho$ associated to $Q$ satisfies the Weyl condition, from the identities (\ref{Diffmeas1}), (\ref{Diffmeas2}) 
and from the fact that our perturbation is only adding real resonances and a finite number of negative eigenvalues eigenvalues $- \frac{\mu_k�}{4}$ for $k=0, ..., N-1$.

\item{Normalization:} We need to verify that $d{\tilde{\rho}}$ satisfies a certain estimate whose formulation requires the introduction of the Hardy-Littlewood maximal function of a measure. 
The argument is more elaborate and we present it in greater detail. Following \cite{KS}, we introduce a measure $d\nu$ on $(1, \infty)$ parametrized by $k$, with $E=k^2$,
\[
\frac{d\nu}{dk}= {\mbox{Im}}\big(M(k^2+i0)\big)-k\,,
\]
which gives on account of (\ref{DiffWT}) with $\kappa=-ik$,
\[
\frac{d{\tilde{\nu}}}{dk}=\frac{d\nu}{dk}-2\sum_{n\geq 0}c_{n}\frac{k}{4k^{2}+\mu_{n}^{2}}\,.
\]
We then define the Hardy-Littlewood maximal function $M_{s}$ of the measure $\nu$ by
\[
(M_{s}\nu)(x)=\sup_{0<L\leq 1}\frac{1}{2L}|\nu|([x-L,x+L])\,.
\]
and compute $|{\tilde{\nu}}|([x-L,x+L])$:
\[
|{\tilde{\nu}}|([x-L,x+L])=|\nu|([x-L.x+L])-2\sum_{n\geq 0}c_{n}\int_{x-L}^{x+L}\frac{k}{4k^{2}+\mu_{n}^{2}}\,dk\,.
\]
Since
\[
\int_{x-L}^{x+L}\frac{k}{4k^{2}+\mu_{n}^{2}}\,dk=\frac{1}{8}\log \frac{(x+L)^{2}+\mu_{n}^{2}}{(x-L)^{2}+\mu_{n}^{2}}\,,
\]
we obtain 
\[
|{\tilde{\nu}}|([x-L,x+L])=|\nu|([x-L.x+L])-2\sum_{n\geq 0}\frac{c_{n}}{8}\log \frac{(x+L)^{2}+\mu_{n}^{2}}{(x-L)^{2}+\mu_{n}^{2}}\,.
\]
Taking $x=k>>1$, the normalization condition we need to verify is 

\begin{equation}\label{norm}
\int_1^{+\infty}\log\bigg[1+\bigg(\frac{(M_{s}{\tilde{\nu}})(k)}{k}\bigg)^{2}\bigg]k^{2}\,dk<\infty\,.
\end{equation}

We have
\[
\log \frac{(k+L)^{2}+\mu_{n}^{2}}{(k-L)^{2}+\mu_{n}^{2}}=\log\big( 1+\frac{4kL}{(k-L)^{2}+\mu_{n}^{2}}\big)=\log(1+{\mathcal{O}}(\frac{L}{k}))\, ,
\]
uniformly in $n$ and $L\in (0,1]$. But by condition ii) in Theorem \ref{KSApp}, we know that the series $\sum_{n\geq 0}|c_{n}|$ is convergent, so it follows that 
\[
|{\tilde{\nu}}|([x-L,x+L])=|\nu|([x-L, x+L])+{\mathcal{O}}(\frac{L}{k})\,,\quad k\to \infty\,,
\]
uniformly in $L\in (0,1]$. Therefore we have 
\begin{equation}\label{HLAsympt}
(M_{s}{\tilde{\nu}})(k)=(M_{s}\nu)(k)+{\mathcal{O}}(\frac{1}{k})\,.
\end{equation}
Now, using the inequality
\[
\log\big(1+(x+y)^{2}\big)\leq C\big(x^{2}+\log(1+y^2)\big)\,,
\]
for some constant $C>0$, we obtain using (\ref{HLAsympt}) the estimate 
\begin{align*}
\log\bigg[1+\bigg(\frac{(M_{s}{\tilde{\nu}})(k)}{k}\bigg)^{2}\bigg]&=\log\bigg[1+\bigg(\frac{(M_{s}{{\nu}})(k)+{\mathcal{O}}(\frac{1}{k})}{k}\bigg)^{2}\bigg] \\
&\leq C \bigg[ {\mathcal{O}}(\frac{1}{k^{4}})+\log\bigg[1+\bigg(\frac{(M_{s}{{\nu}})(k)}{k}\bigg)^{2}\bigg]\bigg]\,,
\end{align*}
which implies the normalization condition (\ref{norm}) after integration over $\mathbb{R}$, using the fact that $\nu$ is associated to a potential $Q\in L^{2}(0,\infty)$.

\item{Lieb-Thirring condition:} The Lieb-Thirring condition $\sum_{j}|{\tilde{E}}_{j}|^{3/2}<\infty$ is trivially satisfied here as was the case for the Weyl condition 
since all we are doing is to add a finite number of negative eigenvalues.

\item{Quasi-Szeg\"o condition:} This condition states that if $d\rho_{0}$ is the free spectral measure, 
that is the spectral measure associated to the zero potential $Q\equiv 0$, then 
\begin{equation}\label{QuasiSzego}
\int_{0}^{\infty}\log\big[\frac{1}{4}\frac{d{\tilde{\rho}}}{d\rho_{0}}+\frac{1}{2}+\frac{1}{4}\frac{d\rho_{0}}{d{\tilde{\rho}}}\big]\sqrt{E}\,dE<\infty\,.
\end{equation}
Again, the verification of this condition for the perturbed measure $d{\tilde{\rho}}$ is more elaborate and we therefore present it in greater detail. 
The spectral measure $d\rho_{0}$ has the expression
\[
d\rho_{0}(E)=\frac{1}{\pi}\,\chi_{(0,\infty)}(E)\sqrt{E}\,dE\,,
\]
so that using (\ref{Diffmeas1}), we have
\begin{equation}\label{Diffmeasmod}
d{\tilde{\rho}}(E)-d{{\rho_{0}}}(E)=d{{\rho}}(E)-d{{\rho_{0}}}(E)-\frac{2}{\pi}\,\sum_{n\geq 0}c_{n}\frac{\sqrt{E}}{4E+\mu_{n}^{2}}\,dE\,.
\end{equation}
We now use (\ref{weaklim}) to express the spectral measure $d\rho$ in terms of the Jost function $\psi(x,\kappa)$ associated to the potential $Q$, 
where we let $E=-\kappa^{2}$. Recall that the Weyl-Titchmarsh function $M$ is given by
\begin{equation}\label{MJost}
M(-\kappa^{2})=\frac{\psi'(0,\kappa)}{\psi(0,\kappa)}\,,
\end{equation}
so that using (\ref{MJost}), we obtain
\begin{align}
\mbox{Im}\,M(-\kappa^{2})&=\frac{1}{2i}\,\bigg(M(-\kappa^{2})-\overline{{M(-\kappa^{2})}}\bigg)=\frac{1}{2i}\,
\bigg(\frac{\psi'(0,\kappa)}{\psi(0,\kappa)}-\frac{{\overline{\psi'(0,\kappa)}}}{{\overline{\psi(0,\kappa)}}}\bigg) \nonumber \\
&=\frac{1}{2i}\,\frac{W(\overline{\psi},\psi)}{|\psi |^{2}}(0,\kappa)\,,\label{ImM}
\end{align}
where $W$ denotes the Wronskian. But $\psi$ and $\overline{\psi}$ are solutions of the same linear second-order 
ODE since the potential $Q$ and the spectral parameter $E=-\kappa^{2}$ are both real. It follows that the Wronskian $W(\overline{\psi},\psi)$ 
is independent of $x$. Now since $\psi$ is the Jost function, we have, in terms of the parameter $k=i\kappa$ introduced in the normalization condition, the asymptotics 
\[
\psi(x,\kappa)\simeq e^{-\kappa x}=e^{ikx}\,, \quad \mbox{for}\,\,x\to \infty\,,
\]
so that 
\[
W(\overline{\psi},\psi)=2ik\,.
\]
Substituting the latter into (\ref{ImM}), we obtain 
\[
\mbox{Im}\,M(-\kappa^{2})=\frac{i\kappa}{|\psi(0,\kappa)|^{2}}\,,
\]
which plugged into in (\ref{weaklim}) gives for $E>0$
\begin{equation}\label{drhoE}
d\rho(E)=\frac{1}{\pi}\,\frac{\sqrt{E}}{|\psi(0,\sqrt{E})|^{2}}\,dE\,.
\end{equation}
But (\ref{Diffmeasmod}) gives
\[
\frac{d{\tilde{\rho}}}{d\rho_{0}}(E)=\frac{d{{\rho}}}{d\rho_{0}}(E)-2\sum_{n\geq 0}c_{n}\frac{1}{4E+\mu_{n}^{2}}\,,
\]
which combined with (\ref{drhoE}) implies
\begin{equation}\label{derivrhotilde}
\frac{d{\tilde{\rho}}}{d\rho_{0}}(E)=\frac{1}{|\psi(0,\sqrt{E})|^{2}}-2\sum_{n\geq 0}\frac{c_{n}}{4E+\mu_{n}^{2}}\,.
\end{equation}
\end{itemize}
We now analyze the asymptotics of $\frac{d{\tilde{\rho}}}{d\rho_{0}}(E)$ in the limit $E\to \infty$. On the one hand we have 
\[
\frac{1}{4E+\mu_{n}^{2}}=\frac{1}{4E}\big(1+\mathcal{O}(\frac{{\mu_n}^{2}}{E})\big)\,,
\]
and on the other hand we know that since the radius of convergence $R$ of the series $\sum_{n\geq 0}c_{n}t^{\lambda}_{n}$ satisfies $R>1$ and since $\lambda_{n}=\mathcal{O}(n)$, 
the series $\sum_{n\geq 0}|c_{n}|\mu_{n}^{2}$ is convergent. The identity (\ref{derivrhotilde}) now implies
\[
\frac{d{\tilde{\rho}}}{d\rho_{0}}(E)=\frac{1}{|\psi(0,\sqrt{E})|^{2}}-\big(\frac{1}{2} \sum_{n\geq 0}c_{n}\big)\frac{1}{E}+\mathcal{O}(\frac{1}{E^{2}})\,.
\]
Using the asymptotics on the modulus of the Jost function given by
\[
|\psi(0,\sqrt{E})|=1+\frac{a}{E}+\mathcal{O}{\frac{1}{E^{2}}}\,,
\]
where $a$ is a real constant, we obtain that 
\begin{equation}\label{derivrhotildeasympt}
\frac{d{\tilde{\rho}}}{d\rho_{0}}(E)=1+\frac{b}{E}+\mathcal{O}(\frac{1}{E^{2}})\,,
\end{equation}
for some real constant $b$, which implies in turn that 
\begin{equation}\label{derivrhoasympt}
\frac{d\rho_{0}}{d{\tilde{\rho}}}=1-\frac{b}{E}+\mathcal{O}(\frac{1}{E^{2}})\,.
\end{equation}
Using (\ref{derivrhotildeasympt}) and (\ref{derivrhoasympt}), we obtain
\begin{equation}\label{Integrand}
\frac{1}{4}\frac{d{\tilde{\rho}}}{d\rho_{0}}+\frac{1}{2}+\frac{1}{4}\frac{d\rho_{0}}{d{\tilde{\rho}}}=1+\mathcal{O}(\frac{1}{E^{2}})\,,
\end{equation}
which implies that the Quasi-Szeg\"o condition (\ref{QuasiSzego}) is satisfied since (\ref{Integrand}) implies that 
\[
\log\big[\frac{1}{4}\frac{d{\tilde{\rho}}}{d\rho_{0}}+\frac{1}{2}+\frac{1}{4}\frac{d\rho_{0}}{d{\tilde{\rho}}}\big]=\mathcal{O}(\frac{1}{E^{2}})\,.
\]

\Section{A few examples : Bargmann potentials}

In this section, we consider perturbations of the potential $Q(x)=0$, (with the associated amplitude function $A(\alpha)=0$), and we give examples of amplitudes $\tilde{A}(\alpha)$ for which we can calculate explicitly the associated  potentials. 
These examples are borrowed from (\cite{GS1}, section 11).

\subsection{First example}
We define for $\alpha \geq 0$,
\begin{equation}
 \tilde{A}(\alpha) = 2 (\gamma^2 - \beta^2) \ e^{-2\gamma \alpha},
\end{equation}
where $\beta>0$ and $\gamma \in [0, \beta[$. Of course, it corresponds to a M\"untz series of the type (\ref{defperturbation0}) (with a single term) 
with $c_0 = 2 (\gamma^2 - \beta^2)<0$ and $\mu_0 = 2\gamma \geq 0$.  

\vskip .1in
\noindent
Thus, we take $\delta = \gamma + \frac{3-d}{2} \geq 3-d$ since $d \geq 3$ by hypothesis. It is known that
\begin{equation}
 \tilde{Q}(x)= -8 \beta^2 \left( \frac{\beta-\gamma}{\beta+\gamma} \right) \ \frac{e^{-2\beta x}}{(1 + \frac{\beta-\gamma}{\beta+\gamma} \  e^{-2\beta x})^2}
\end{equation}
The associated Jost function is given in the variable $\kappa = -ik$ by
\begin{equation}
 \psi(0, \kappa) = \frac{\kappa+\gamma}{\kappa + \beta},
\end{equation}
(see \cite{GS1}, case 2, p. 636) and is holomorphic in $\mbox{Re} \ \kappa > -\beta$.  The unique root of the Jost function is given by $\kappa=-\gamma$ which is a real resonance.

\subsection{Second example}
For $\alpha>0$, we define  the amplitude
\begin{equation}
 \tilde{A}(\alpha) = - \frac{2c_1}{\kappa_1} \ \sinh (2\kappa_1 \alpha),
\end{equation}
where $c_1>0$ is a normalization constant and $\kappa_1>0$. It corresponds to a M\"untz series with two terms and with two $\mu_k$ of different sign. The associated potential is given by
\begin{equation}
 \tilde{Q}(x)=  -2 \frac{d^2}{dx^2} \ \log \left( 1 + \frac{c_1}{\kappa_1^2} \ \int_0^x \sinh^2 (\kappa_1 y ) \ dy \right),
\end{equation}
ahe Jost function has the form in the $\kappa$ variable
\begin{equation}
	\psi(0, \kappa) = \frac{\kappa-\kappa_1}{\kappa + \kappa_1},
\end{equation}
(see \cite{GS1}, case 1, p. 635). We note that the Jost function is vanishing at $\kappa = \kappa_1$ wich corresponds to the single negative eigenvalue $-\kappa_1^2$.

\vskip .1in
\noindent

 
\Section{Gel'fand-Levitan equations and local stability estimates}

\subsection{Proof of Theorem \ref{Mainest}.}

\vskip .1in
\noindent

In this section, we deduce from the estimates for the difference of the amplitudes $A -\tilde{A}$ obtained in Section 4 a set of new H\"older {\it local} stability estimates for the difference of the associated potentials $Q - \tilde{Q}$. By local stability, we mean that we are able to control the norm   $|| Q - \tilde{Q} ||_{L^2(0,T)}$ with respect to $\epsilon$, if the Steklov spectra of the underlying Schr\"odinger operators are close up to $\epsilon$ as in (\ref{StabilityStatementI}), $T$ being any fixed positive parameter. 

\vskip .1in
\noindent

More precisely, we assume here that the potential $Q \in L^2(0,\infty)$ (and thus its associated amplitude $A$), is {\it{fixed}} and that $\tilde{Q} \in L^2(0,\infty)$  belongs to the infinite dimensional class, denoted $\mathcal{C}_Q$,  defined above, that is we assume that the associated amplitude to  $\tilde{Q}$  has the form
\begin{equation}
	{\tilde{A}}(\alpha)=A(\alpha)+\sum_{k\geq 0}c_{k}e^{-\mu_k \alpha}\,,\quad \alpha >0\,,
\end{equation} 
where $\mu_k = 2k+d-3 +2\delta$ and $\delta \geq 3-d$. Moreover, we assume that $ c_k \leq 0$ for all $k \geq 0$ and the power series $\sum_{k\geq 0}c_{k} t^{\lambda_k}$ has a radius of convergence $R>1$.

\vskip .1in
\noindent
To obtain these local stability estimates, we shall make intensive use of the local version of the classical Gel'fand-Levitan equations, (see for instance  (\cite{ABM}, Eq. (2.24)) which we recall here. For $0 \leq x \leq t \leq T$, we consider the integral equation 
\begin{equation} \label{LGL}
V(x,t) + \int_x^T K(t,s) V(x,s) \ ds = -K(x,t)	,
\end{equation}
where the integral kernel $K(t,s)$ is given by
\begin{equation}\label{noyau}
	K(t,s) = p(2T-t-s) - p(|t-s|),
\end{equation}
and 
\begin{equation}\label{defp}
	p(t) = - \half \int_0^{\frac{t}{2}} A(\alpha) \ d \alpha .
\end{equation}

\vskip .1in
\noindent
These integral equations are uniquely solvable  for all $ x \in (0,T)$ and we can recover the underlying potential using the relation:
\begin{equation}\label{solutionQ}
	Q(T-x) = -2 \frac{d}{dx} V(x,x)\,.
\end{equation}
An easy calculation shows that
\begin{eqnarray}\label{derivee}
	\frac{d}{dx} \left( V(x,x) \right) &=& p(2T-x) V(x,x) + 2 p'(2T-2x) - \int_x^T 	\left( p(2T-x-s) -p(s-x) \right)\  \frac{\partial V}{\partial x} (x,s) \ ds \nonumber \\
& & - \int_x^T 	\left( p'(2T-x-s) -p'(s-x) \right) \ V(x,s) \ ds.
\end{eqnarray}	
	
\vskip .1in
\noindent
Let us begin with an elementary result:

\begin{lemma}\label{estdiffp}
	Under the hypotheses of Theorem \ref{Mainest}, there exists a constant $C_T$ depending only on $T$ such that
	\begin{equation}
		|| p - \tilde{p} ||_{ ( C^0(0,2T), || \cdot ||_{\infty} )} \leq C_T \ f(\epsilon),
	\end{equation}
where
\begin{equation}
f(\epsilon) = \left( B^2 \epsilon + R^{1-d} \ \e^{\frac{\log R}{\log (\frac{9M_0}{2})}} \right)^{\half}
\end{equation}
\end{lemma}

\begin{proof}
	This is  an immediate application of (\ref{l2estimateA}) and the Cauchy-Schwartz inequality.
\end{proof}

\vskip .1in
\noindent
Now, let us introduce some notation to simplify the presentation below. In what follows, the parameters $x$ and $T$ are assumed to be fixed and $t$ is a variable lying in the interval $[x,T]$. We denote by $K$ the integral operator on $L^2(x,T)$ with kernel $K(t,s)$, 
\begin{equation}
	Kf(t) = \int_x^T K(t,s) \ f(s)  \ ds,
\end{equation}
and set
\begin{equation}
	d(t) := p(t-x)-p(2T-x-t).
\end{equation}
Thus, the solution  $V(x,.)$ of the integral equation (\ref{LGL}) can be written as

\begin{equation} \label{V}
	V:= V(x,.) = (I+K)^{-1} d.
\end{equation}
Using (\ref{V}) and the usual resolvent identity, one obtains
\begin{equation}\label{resolvent}
	\tilde{V}-V = (I+\tilde{K})^{-1} \left( \tilde{d} - d  + (K - \tilde{K}) (I+K)^{-1}d \right)
\end{equation}
By Lemma \ref{estdiffp}, one has the uniform estimate  for $t,s \in [0,T]$,
\begin{equation}
 | \tilde{d}(t) - d(t) | \leq C_T \ f(\epsilon) \ \  ,\ \ 	| \tilde{K}(t,s) - K(t,s) | \leq C_T \ f(\epsilon),
\end{equation}
thus using Schur's lemma, one gets
\begin{equation}\label{schur}
	|| \tilde{K} - K || \leq C_T \ f(\epsilon),
\end{equation}
in the sense of the operator norm  on $L^2(x,T)$. As a consequence for $\epsilon>0$ sufficiently small, the operator $I+ (I+K)^{-1} (\tilde{K} -K)$ is invertible, and using again the resolvent identity, one obtains easily
\begin{equation}
(I+\tilde{K})^{-1} = \left( I+ (I+K)^{-1} (\tilde{K} -K) \right)^{-1} (I+K)^{-1}.
\end{equation}
It follows that, for $\epsilon <<1$, the operator norm of $(I+\tilde{K})^{-1}$ is uniformly bounded:
\begin{equation} \label{estop}
	|| (I+\tilde{K})^{-1} || \leq  2  \ || (I+K)^{-1}||.
\end{equation}
Thus, thanks to (\ref{resolvent}) ,  (\ref{schur}) and (\ref{estop}), one has:
\begin{equation}\label{estV}
	|| \tilde{V}- V ||_{L^2 (x,T)} \leq C_T \ f(\epsilon)
\end{equation}

\vskip .1in
\noindent
In the same way, differentiating the integral equation (\ref{LGL}) with respect to $x$, one obtains:
\begin{equation}
	\frac{\partial V}{\partial x} (x,t)+ \int_x^T K(t,s) 	\frac{\partial V}{\partial x}(x,s) = -p'(t-x) +p'(2T-x-t) + 
	K(t,x) V(x,x),
\end{equation}
and by the same argument, we get immediately
\begin{equation}\label{estderV}
||  \frac{\partial \tilde{V}}{\partial x} - \frac{\partial V}{\partial x}          ||_{L^2 (x,T)} \leq C_T \  f(\epsilon).
\end{equation}	
Finally, using (\ref{derivee}), (\ref{estV}) and (\ref{estderV}), mimicking the above arguments, one has for all $0 \leq x \leq T$,
\begin{equation}
	||  \frac{d}{dx} \left( \tilde{V}(x,x) \right) - \frac{d}{dx} \left( V(x,x) \right)      ||_{L^2 (x,T)} \leq C_T \  f(\epsilon).
\end{equation}
Then taking $x=0$ and using (\ref{solutionQ}), we see that
\begin{equation}
||  \tilde{Q}- Q ||_{L^2 (x,T)} \leq C_T \  f(\epsilon),
\end{equation}
and the proof of Theorem \ref{Mainest} is complete.

\subsection{Proof of Corollary \ref{Calderon1}.}

First, it is easy to see that  
\begin{equation}\label{bounded}
\Lambda_q - \Lambda_{\tilde{q}} \in B(L^2(S^{d-1}))\,.
\end{equation}
Indeed, the restriction of the DN map onto the orthogonal projection of the restriction to $S^{d-1}$ of the space of homogeneous harmonic polynomials of degree $k$ in $\mathbb{R}^d$ satisfies
$$
\Lambda_q^k \psi_k = \sigma_k \psi_k .
$$
Thus,
\begin{eqnarray*}
|| (\Lambda_q - \Lambda_{\tilde{q}}) \psi ||_{L^2}^2 &=& || \sum_k (\sigma_k -\tilde{\sigma_k}) \psi_k Y_k ||_{L^2}^2 \\
                                                     &=& \sum_k |\sigma_k -\tilde{\sigma_k}| |\psi_k|^2 \\
                                                     & \leq & || \sigma_k -\tilde{\sigma_k}||_{l^{\infty}(\N)} \  || \psi||_{L^2}.
\end{eqnarray*}
So, we deduce that $||\Lambda_q - \Lambda_{\tilde{q}}||_{B(L^2(S^{d-1}))} = || \sigma_k -\tilde{\sigma_k}||_{l^{\infty}(\N)}<\infty$ and thus that (\ref{bounded}) holds. It follows that the local H\"older stability estimates obtained in Theorem \ref{Mainest} imply that for any $T>0$, there exists a positive constant $C_T$ such that 
$$
|| Q-\tilde{Q}||_{L^2 (0,T)} \leq C_T \ ||\Lambda_q - \Lambda_{\tilde{q}}||_{B(L^2(S^{d-1}))}^{\theta},
$$
or equivalently,
$$
|| q-\tilde{q}||_{L^2 ((e^{-T}, 1), r^3 dr)} \leq C_T \ ||\Lambda_q - \Lambda_{\tilde{q}}||_{B(L^2(S^{d-1}))}^{\theta}.
$$

\vspace{0.8cm}
\noindent \textbf{Acknowledgements}: The authors would like to warmly thank the anonymous referees for their valuable comments and suggestions.


\end{document}